\documentclass[
final
,nomarks
]{dmtcs-episciences}
\usepackage[utf8]{inputenc}
\usepackage{amsmath}
\usepackage{amsfonts}
\usepackage{latexsym}
\usepackage{amssymb}
\usepackage{tikz}
\usetikzlibrary{arrows}
\usepackage{mathrsfs}
\usepackage{amsthm}
\usepackage{todonotes}
\usetikzlibrary{positioning}
\usepackage{enumerate}
\usepackage{stmaryrd}
\usepackage{xparse}
\usepackage{subcaption}

\newcommand{\set}[1][r]{\llbracket 1, #1 \rrbracket}
\newcommand{\comp}[2][\lambda]{(#1_1,\dots,#1_#2)}
\newcommand{\comptrois}{(\lambda_1,\lambda_2,\lambda_3)}
\newcommand{\compmoins}[1][t]{(\lambda_1,\dots,\lambda_{#1} - |s_{#1}|, \dots, \lambda_r)}
\newcommand{\compplus}[1][t]{(\lambda_1,\dots,\lambda_{#1} + |s_{#1}|, \dots, \lambda_r)}
\newcommand{\compenlever}[1][t]{(\lambda_1,\dots, \lambda_{#1-1},\lambda_{#1+1},\dots,\lambda_r)}

\newcommand{\perm}{\mathcal{S}}

\newcommand{\N}{\mathbb{N}}
\newcommand{\Z}{\mathbb{Z}}
\newcommand{\cycle}{f}

\newcommand{\DIET}{symmetric discrete interval exchange}
\newcommand{\DIETtrois}{symmetric discrete 3-interval exchange}

\newtheorem{definition}{Definition}
\newtheorem{exemple}{Example}
\newtheorem{theorem}{Theorem}
\newtheorem{lemma}{Lemma}

\newtheorem{corollary}{Corollary}

\providecommand{\keywords}[1]{\textbf{Keywords --} #1}
\author{M\'{e}lodie Lapointe}
\title{Number of Orbits of Discrete Interval Exchanges}
\affiliation{LaCIM, Universit\'e du Qu\'ebec \`a Montr\'eal, Canada}
\received{2018-11-1}
\revised{2019-4-29}
\accepted{2019-5-13}

\begin{document}
\publicationdetails{21}{2019}{3}{17}{4951}
\maketitle

\begin{abstract}
    A new recursive function on discrete interval exchange associated to a composition of length $r$, and the permutation $\sigma(i) = r -i +1$ is defined. Acting on a composition $c$, this recursive function counts the number of orbits of the discrete interval exchange associated to the composition $c$. Moreover, minimal discrete interval exchanges, \textit{i.e.} the ones having only one orbit, are reduced to the composition $(1,1)$ which label the root of the Raney tree. Therefore, we describe a generalization of the Raney tree using our recursive function.
\end{abstract}

\section{Introduction}

 An interval exchange is a map which cuts an interval into $r$ pieces and reorders them. 
These maps were introduced by Oseledec in 1966~\cite{O1966} to generalize circle rotations following some ideas of Arnold~\cite{A1963,A2004}.
The discrete analogue maps were introduced by Ferenczi and Zamboni in 2013~\cite{FZ2013} by replacing the interval by a discrete interval \textit{i.e.} a set of integers between $a$ and $b$.
A discrete interval exchange is called minimal if it has only one orbit.
In the same paper, Ferenczi and Zamboni give a bijection between minimal discrete interval exchange and a special subset of words called $\pi$-clustering words (see~\cite{FZ2013} for more details on $\pi$-clustering words).
A natural question arises from that bijection: Can we identify compositions associated with minimal discrete interval exchanges?

Some partial answers can be found in the literature when the pieces are reordered in reverse order.
It is well-known that a discrete interval exchange with 2 pieces is minimal if and only if its composition has coprime parts.
Moreover, a discrete interval exchange with 3 pieces $\comptrois$ is minimal if and only if $\gcd(\lambda_1 + \lambda_2, \lambda_2 + \lambda_3) = 1$ as shown by Pak and Redlich~\cite{PR2008}.
We propose a new recursive function on compositions which counts the number of orbits of a discrete interval exchange and agrees with the greatest common divisor recursion when $r =2$.
Other functions that coincide with the greatest common divisor when $r = 2$ are known. 
For example, Rauzy induction is a generalization of Euclid's algorithm which renormalized intervals exchange and was introduced by Rauzy in 1979~\cite{R1979}. Other examples are given by the multidimensional continued fraction algorithms (see~\cite{L2015,S2000}).  

This paper is organized as follows. 
In Section~\ref{sec:def}, a formal definition of symmetric discrete interval exchanges is presented.
The presentation of the recursive function counting the number of orbits is the main focus of the Section~\ref{sec:counting}.
In Section~\ref{sec:tree}, a generalization of the Raney tree where each circular composition appears is shown.
Finally, we describe the number of orbits of any symmetric discrete interval exchange with 3 pieces extending the result given  by Pak and Redlich~\cite{PR2008}. Moreover, we give the cyclic type of these discrete interval exchanges.

Note that the recursive function counting the number of orbits presented in Theorem~\ref{thm:algo}, Corollary~\ref{cor:cycle_type_3} giving the cyclic type of a discrete interval exchange of length $3$ and a slightly modified version of the tree of circular compositions were conjectured by Christophe Reutenauer~\cite{R2016}. 

After submission, we discovered that Karnauhova and Liebscher~\cite{KL2017} (2017) have given a similar algorithm as ours (Theorem~\ref{thm:algo}). They work on a geometric generalization of symmetric discrete interval exchanges, called bi-rainbow meanders. Their main contribution is a formula which counts the number of connected components of bi-rainbow meanders; this is the same as counting the number of orbits of symmetric discrete interval exchanges. They do not study the construction, nor the enumeration of connected bi-rainbow meanders, which corresponds in our language to minimal symmetric discrete interval exchanges. 

\section{Discrete interval exchange}~\label{sec:def}

An \emph{integer interval}, denoted by $\llbracket a,b\rrbracket$, is the set $[a,b] \cap \Z$ for some integers $a$ and $b$.
A \emph{composition} of $n$ is a finite sequence $\comp{r}$ of positive integers whose sum is $n$. 
A \emph{permutation} $\sigma$ of $\perm_r$ can be written as the \emph{word} $\sigma(1)  \sigma(2)  \dots  \sigma(r).$
Recall that the \emph{cycle notation} of a permutation $\sigma$ is written as the product of its orbits.
For example, the permutation written as the word $451326$, is written in cycle notation as $(1,4,3)(2,5)(6)$. 
The number of orbits of $\sigma$ is denoted by $\gamma(\sigma)$.
A permutation is \emph{circular} if it has exactly one orbit. 

We consider a particular case of \emph{discrete interval exchanges}, as they are defined in~\cite{FZ2013} (as they are defined in~\cite{FZ2013}: for $\pi \in S_r$, let $\pi(k) = r-k+1$ for all $k \in \set[r]$). 
\begin{definition}~\label{D:DIET}
    Let $\lambda = \comp{r}$ be a composition of $n$. 
    The associated \emph{symmetric discrete interval exchange} $T_{\lambda}$ is defined as follows:
    let $B_1,\dots,B_r$ be the subintervals of $\set[n]$ defined by
    $    B_i = \llbracket 1+ \sum_{j < i} \lambda_j ,  \sum_{j \leq i} \lambda_j \rrbracket;
    $
let 
    \begin{align*}
        s_i = \sum_{j > i} \lambda_j - \sum_{j < i} \lambda_j.
    \end{align*}
    Then for all $x \in B_i$, $T_{\lambda}(x) = x+s_i$.
    We call $s_i$ the $i$-th \emph{translation parameter} and $(s_1,\dots,s_r)$ the \emph{translation vector}.
\end{definition}

The symmetric discrete interval exchange $T_{\lambda}$ is a permutation on the set $\set[n]$.
For example, the symmetric discrete interval exchange $T_{(3,5,4,2)}$ is $(1,12,6,9,3,14,2,13)(4,7,10)(5,8,11)$ as shown in Figure~\ref{fig:DIET}. The subintervals of $\set[14]$ are $B_1 = \{1,2,3\}$, $B_2= \{4,5,6,7,8\}$, $B_3 = \{9,10,11,12\}$, $B_4 = \{13,14\}$ and the translation vector of $\set[14]$ is $s = (5+4+2,4+2-3,2-3-5,-3-5-4) = (11,3,-6,-12)$. 
A discrete interval exchange is \emph{minimal} if it has only one orbit.

The translation vector of a \DIET \ is strictly decreasing. Moreover,  
the inverse of a symmetric discrete interval exchange $T_{\comp{r}}$ is $T_{(\lambda_r,\dots, \lambda_1)}$. Evidently, they both have the same number of orbits. 

\begin{figure}
    \begin{center}
    \begin{tikzpicture}[scale=0.8]
        \foreach \x in {1,2,...,14}
        {
            \draw node (\x) at (\x,0) {$\x$};
        }
        \foreach \x in {1,2,...,14}
        {
            \draw node (\x) at (\x,-2) {$\x$};
        }
        \draw[rounded corners=8,thick,green] (0.6,-0.4) rectangle (3.4,0.4);
        \draw[rounded corners=8,thick,blue] (3.6,-0.4) rectangle (8.4,0.4);
        \draw[rounded corners=8,thick,red] (8.6,-0.4) rectangle (12.4,0.4);
        \draw[rounded corners=8,thick,orange] (12.6,-0.4) rectangle (14.4,0.4);
        \draw[rounded corners=8,thick,green] (11.6,-1.6) rectangle (14.4,-2.4);
        \draw[rounded corners=8,thick,blue] (6.6,-1.6) rectangle (11.4,-2.4);
        \draw[rounded corners=8,thick,red] (2.6,-1.6) rectangle (6.4,-2.4);
        \draw[rounded corners=8,thick,orange] (0.6,-1.6) rectangle (2.4,-2.4);
        \draw[thick,->,green] (2,-0.4) -- (13,-1.6);
        \draw[thick,->,blue] (6,-0.4) -- (9,-1.6);
        \draw[thick,->,red] (10.5,-0.4) -- (4.5,-1.6);
        \draw[thick,->,orange] (13.5,-0.4) -- (1.5,-1.6);
        \draw node at (2,0.7) {$B_1$};
        \draw node at (6,0.7) {$B_2$};
        \draw node at (10.5,0.7) {$B_3$};
        \draw node at (13.5,0.7) {$B_4$};
        \draw node at (13,-2.7) {$B_{\sigma(1)}$};
        \draw node at (9,-2.7) {$B_{\sigma(2)}$};
        \draw node at (4.5,-2.7) {$B_{\sigma(3)}$};
        \draw node at (1.5,-2.7) {$B_{\sigma(4)}$};
    \end{tikzpicture}  
    \end{center}
    \caption{The symmetric discrete interval exchange $T_{(3,5,4,2)}$}
    \label{fig:DIET}
\end{figure}
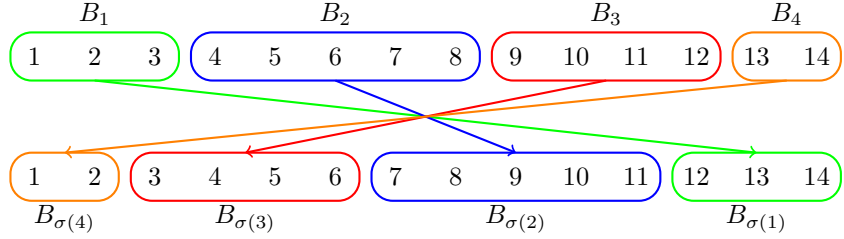

%*****************************
%Counting the number of cycles
%*****************************
\section{Counting Orbits}~\label{sec:counting}

In this section, we define a recursive function on compositions that counts the number of orbits of a symmetric discrete interval exchange. In fact, the function subtracts the $i-$th translation parameter to $\lambda_i$. To ensure that the result is positive, we first prove that each composition has a part which is greater than its translation parameter.

\begin{lemma}~\label{l:existence}
    Let $\lambda = \comp{r}$ be a composition. There exists an integer $t$ in the set $\set$ such that $\lambda_t \geq |s_t|$.
\end{lemma}

\begin{proof}
   There exists an integer $t$ such that 
    \begin{align*}
        \sum_{i < t} \lambda_i \leq \sum_{i \geq t } \lambda_i & \text{ and } 
        \sum_{i \leq t} \lambda_i \geq \sum_{i > t} \lambda_i,
    \end{align*}
    since all $\lambda_i$ are positive integers.
    By isolating $\lambda_{t}$ in both inequalities, we obtain 
    \begin{align*}
        \sum_{i < t} \lambda_i - \sum_{i > t } \lambda_i \leq \lambda_t& \text{ and } 
        \lambda_t \geq  \sum_{i > t} \lambda_i - \sum_{i < t} \lambda_i ,
    \end{align*}
    
    We combine both inequalities to obtain  
    \begin{align*}
        \lambda_t \geq  \left| \sum_{i > t} \lambda_i - \sum_{i < t} \lambda_i \right|  = |s_t|,
    \end{align*}
    by Definition~\ref{D:DIET}.

\end{proof}

\textbf{Remark: }The value $t$ where $\lambda_t \geq |s_t|$ is not unique if $\lambda_t = |s_t|$.
For example, in the composition $(2,1,2,1)$ we have that $\lambda_2 \geq |s_2|$ and $\lambda_3 \geq |s_3|$.
However, one can show that there are at most two possible values. These properties are proven in Section~\ref{sec:tree}.

Let $t$ be the smallest integer belonging to the set $\set$ such that $\lambda_t\geq |s_t|$ where $(s_1,\dots,s_r)$ is the translation vector of $T_\lambda$.
Define a function $\cycle$ mapping from the set of compositions to the natural numbers as
\begin{equation}
    \cycle\comp{r} = \begin{cases}
                    \lambda_1 & \text{if } r = 1, \\
                    \lambda_t + \cycle\compenlever  & \text{if } |s_t| = 0,\\
                    \cycle\compenlever &  \text{if } \lambda_t = |s_t|, \text{ and}\\
                    \cycle\compmoins &   \text{otherwise}.\\
                    \end{cases}
    \label{eqn:formule}
\end{equation}
We show that $f$ counts the number of orbits of any \DIET  \ i.e. $f(\lambda) = \gamma(T_{\lambda})$.
\begin{exemple}
    The \DIET \ $T_{(3,5,4,2)}$ has $3$ cycles as shown previously and 
    \begin{align*}
        f(3,5,4,2) & = f(3,2,4,3) = f(3,2,1,2) \\
                   & = 2 + f(3,1,2) = 2 + f(1,2) = 2 + f(1,1) = 2 + f(1)\\ 
                   & = 3.
    \end{align*}
    The values of $t$ and the corresponding translation parameter at each step are  in Table~\ref{t:calcul_f}. 
\end{exemple}

\begin{table}
    \centering
    \begin{tabular}{c|c|c|c|c|c|c|c}
        $\lambda$ & $(3,5,4,2)$  & $(3,2,4,2)$ & $(3,2,1,2)$ & $(3,1,2)$  & $(1,2)$ & $(1,1)$ & $(1)$  \\
        \hline
        $t$ & 2 & 3 & 2 & 1 & 2 & 1 &  \\
        \hline
        $|s_t|$ & 3 & 3 & 0 & 3 & 1 & 1 &  \\
    \end{tabular}
    \caption{The first row gives the successive steps in the computation of $f(3,5,4,2)$, the second row gives the indices where the recursion for $f$ is applied on the composition and the third row gives the absolute value of the $t$-th translation parameter.}
    \label{t:calcul_f}
\end{table}

We say that
$\comp{r} \leq (\beta_1, \dots, \beta_s)$
if $r < s$ or $r = s$ and $\comp{r} \leq_{lex} (\beta_1, \dots, \beta_s)$, where $\leq_{lex}$ is the lexicographic order i.e. $\lambda_1 < \beta_1$ or $\lambda_1 = \beta_1$ and $(\lambda_2,\dots,\lambda_r) \leq_{lex} (\beta_2,\dots,\beta_s)$. 
One can check that the composition in the recursive step of $f$ is smaller than the composition.

The next few lemmas show that the operations used in the recursive step of the function $f$ are valid.

\begin{lemma}~\label{l:addcycles}
    If the value of $|s_t|$ is zero , then the \DIET \ $T_{\lambda}$ has
$$\lambda_t + \gamma(T_{\compenlever})$$
orbits.
\end{lemma}

For example, $T_{(3,2,4,5)}$ has 5 orbits and its third translation parameter is $0$. Moreover, one can check that $T_{(3,2,5)}$ has one orbit, hence $\gamma(T_{(3,2,4,5)}) = 5 = 4 + \gamma(T_{(3,2,5)})$
as predicted.

\begin{proof}
    Let $\lambda = \comp{r}$ be a composition and $t$ an integer in $\set$ such that $s_t = 0$.
    Let $x$ be an integer in the set $\set[n]$. 
    By definition, the \DIET \ $T_{\lambda}$ is 
    \begin{align}~\label{eq:proof}
        T_{\lambda}(x) = x + s_i
    \end{align}
    where $x$ belongs to $B_i$  with $B_i = \llbracket 1 + \sum_{j < i} \lambda_j,\sum_{j \leq i } \lambda_j \rrbracket $ and $s_i = \sum_{j > i} \lambda_j - \sum_{j < i} \lambda_j$ as in Definition~\ref{D:DIET}.

    If $x \in B_t$, then Equation~\eqref{eq:proof} becomes $T_{\lambda}(x) = x$. Thus, each integer in $B_t$ is a fixed point of $T_{\lambda}$ and its orbit contains only itself. Thus $T_{\lambda}|B_t$ is the identity and $T_{\lambda}$ stabilizes $H = \set[n]\backslash B_t$. 
    
    We show that $T_{\lambda}|H$ is conjugated to $T_{\lambda'}$; that is , $T_{\lambda} | H = \omega^{-1}\circ T_{\lambda'} \circ \omega$, where $\omega$ is the unique decreasing bijection $H \rightarrow \set[n-\lambda_t]$. This implies that the number of orbits of $T_{\lambda} | H$ is equal to $\gamma(T_{\lambda'})$. 

    Suppose that $x \notin B_t$.
    Let $\lambda' = (\lambda'_1,\dots,\lambda'_{r-1})$ where $\lambda'_i = \lambda_i $ if $i < t$ and $\lambda'_i = \lambda_{i+1}$ otherwise.
    Note that $T_{\lambda'}$ is the \DIET \ where $B'_1,\dots, B'_{r-1}$ are subintervals of $\set[n - \lambda_t]$ defined by $B'_i = \llbracket 1 + \sum_{j < i} \lambda'_i, \sum_{j \leq i} \lambda'_i \rrbracket$. Its translation parameters are defined by $s'_i = \sum_{j > i} \lambda'_j - \sum_{j < i} \lambda'_j$.
    Then for all $x \in B'_i$, $T_{\lambda'}(x) = x + s'_i$.
    Note that $T_{\lambda'} = T_{\compenlever}$.

    First, we describe the sets $B_i$ using the sets $B'_i$. 
    We have two cases to consider. 

    If $i < t$, then 
    $$B_i = \left\llbracket 1 + \sum_{j < i} \lambda_j, \sum_{j \leq i} \lambda_j \right\rrbracket = \left\llbracket 1 + \sum_{j < i} \lambda'_j, \sum_{j \leq i} \lambda'_j \right\rrbracket = B'_i$$
    since $j \leq i < t$, we know that $\lambda_j = \lambda'_j$.

    Otherwise, if $i > t$, then 
    $$B_i = \left\llbracket 1  + \sum_{j < i} \lambda_j, \sum_{j \leq i} \lambda_j \right\rrbracket = \left\llbracket 1 + \lambda_t + \sum_{j < i} \lambda'_i, \lambda_t + \sum_{j \leq i} \lambda'_i \right\rrbracket = \{x + \lambda_t \mid x \in B'_i\},$$
    since $i > t$ and $\sum_{j < i} \lambda_j = \sum_{j < t} \lambda_j + \sum_{t < j < i} \lambda_j + \lambda_t  = \lambda_t + \sum_{i < j} \lambda'_i$.

    Secondly,
    we want to show that $T_{\lambda}(x)  = T_{\lambda'}(x) \pm \lambda_t$. Recall that $T_{\lambda}(x)= x + \sum_{j > i } \lambda_j - \sum_{j < i } \lambda_j$. We factorize $\lambda_t$ to obtain
    $$T_{\lambda}(x) = \begin{cases}
        x + \sum_{j > t } \lambda_j + \sum_{t > j > i }  \lambda_j- \sum_{j < i }  \lambda_j + \lambda_j & \text{ if } i < t \\
        x + \sum_{j > i } \lambda_j - \sum_{j < t }  \lambda_j- \sum_{t < j < i }  \lambda_j - \lambda_j & \text{ if } i > t.
        \end{cases}$$
        Replacing $\lambda_j$ by $\lambda'_j$ if $j < t$ and by $\lambda'_{j-1}$ otherwise gives 
    \begin{align*}
        T_{\lambda}(x) & = \begin{cases}
        x + \sum_{j > i }  \lambda'_j- \sum_{j < i }  \lambda'_j + \lambda_t  & \text{ if } i < t \\
        x + \sum_{j > i } \lambda'_j - \sum_{j < i }  \lambda'_j + \lambda_t  & \text{ if } i > t,
        \end{cases} \\
        & = \begin{cases}
        x + s_i + \lambda_t & \text{ if } i < t \\
        x + s_i  - \lambda_t & \text{ if } i > t.
        \end{cases}
    \end{align*}
   
    Hence, $T_{\lambda}|H$ is conjugated to $T_{\lambda'}$ and the number of orbits of $T_{\lambda}$ is $\lambda_t + \gamma(T_{\compenlever})$ as claimed. 
\end{proof}

\begin{lemma}~\label{l:diffcycles}
    Let $ \comp{r}$ be a composition of $n$. 
    The \DIET \ $T_{\comp{r}}$ and $T_{\compplus}$ have the same number of orbits. 
\end{lemma}

We begin by understanding how the orbits of $T_{\comp{r}}$ can be transformed into the orbits of $T_{\compplus}$. Through the cycle notation, a \DIET \ can be viewed as a set of circular words on the alphabet $\set[n]$. Therefore, a  substitution on $T_{\comp{r}}$ describes this operation:
\begin{align}
    \psi_t(x) =  \begin{cases} x & \mathrm{ if } \ x \in \set[k]\\
        x \cdot (x + |s_t|) & \mathrm{ if } \ x \in \llbracket k + 1, k + |s_t|\rrbracket \mathrm{\ and\ } s_t > 0\\
        (x + |s_t|)\cdot x & \mathrm{ if } \ x \in \llbracket k+1 , k + |s_t|\rrbracket \mathrm{\ and\ } s_t < 0\\\
        x + |s_t| & \mathrm{ otherwise } \\
    \end{cases}
    \label{def:morphism}
\end{align}
where $x \in \set[n]$ is a letter, $\cdot$ denotes concatenation and 
$$k = \begin{cases} \sum_{j\leq t} \lambda_j & \mathrm{ if \ } s_t \geq 0 \\
    \sum_{j < t} \lambda_j - |s_t| & \mathrm{ if } s_t < 0.
\end{cases}$$

For example, the \DIET \ of $T_{(5,1,2)}$ is $(1, 4, 7)(2, 5, 8)(3, 6)$ and the \DIET \ of $T_{(5,4,2)}$ is $(1, 7, 4, 10)(2, 8, 5, 11)(3, 9, 6)$. Let us apply $\psi_2$ on $T_{(5,1,2)}$. We have that $t = 2$, $s_t = -3 < 0$ and $k = 2$. For $T_{(5,1,2)}$, the substitution is 
\begin{align}
    \psi_2(x) =  \begin{cases} x & \mathrm{ if } \ x \in \set[2]\\
        (x + 3)\cdot x & \mathrm{ if } \ x \in \llbracket 3 , 5\rrbracket \\
        x + 3 & \mathrm{ otherwise. } \\
    \end{cases}
\end{align}
 
Hence $\psi_2(T_{(5,1,2)}) = \psi_2((1, 4, 7)(2, 5, 8)(3, 6))  = (1, 7, 4, 10)(2, 8, 5, 11)(6, 3, 9)$ as predicted.
\begin{lemma}~\label{l:morphism}
    Let $t$ be an integer in $\set[r]$.
        The identity
    $$\psi_t(T_{\comp{r}}) = T_{\compplus}$$
     holds.

\end{lemma}

Here $\psi_t$ acts as a substitution on the orbits of $T_{\comp{r}}$.
Recall that the circular factors of length 2 of a circular word $(a_1,\dots, a_n)$ are the words $a_ia_{(i+1)\bmod n}$, with $i \in \set[n]$.
We use the following obvious but useful result in the next proof: 
the orbits of $\sigma$ and $\alpha$ have the same circular factors of length 2.

\begin{proof}
    First, suppose that $s_t = 0$. Thus, $\psi_t(T_{\lambda}) = T_{\compplus}$, since $\psi_t$ is the identity morphism when $s_t = 0$.

    Let $t$ be an integer in the set $\set[r]$
    and $T_{\lambda}$ a \DIET \ as in Definition~\ref{D:DIET}.
    Let $\lambda' = (\lambda'_1,\dots,\lambda'_r)$ be a composition where $\lambda'_t = \lambda_t + |s_t|$ and $\lambda'_i = \lambda_i$ if $t \neq i$.
    Recall that $T_{\lambda'}$ is the \DIET \ where $B'_1,\dots, B'_r$ are subintervals of $\set[n + |s_t|]$ defined by $B'_i = \llbracket 1 + \sum_{j < i} \lambda'_i, \sum_{j \leq i} \lambda'_i \rrbracket$ and its translation parameters are $s'_i = \sum_{j > i} \lambda'_j - \sum_{j < i} \lambda'_j$.
    Then for all $x \in B'_i$, $T_{\lambda'}(x) = x + s'_i$.
    Note that the previous definition implies $T_{\lambda'_t} = T_{\compplus}$.

    We begin by describing the circular factors of length $2$ of $T_{\lambda'}$ using $\lambda$ and $B_i$ instead of $\lambda'$ and $B'_i$.

    To describe $B'_i$ using the sets $B_i$, we have three cases to consider. 
    First, if $i < t$, then 
    $$B'_i = \left\llbracket 1 + \sum_{j < i} \lambda'_i, \sum_{j \leq i} \lambda'_i \right\rrbracket = \left\llbracket 1 + \sum_{j < i} \lambda_i, \sum_{j \leq i} \lambda_i \right\rrbracket = B_i$$
    since $j \leq i < t$, we know that $\lambda'_j = \lambda_i$.
    Secondly, if $i = t$, then 
    $$B'_t = \left\llbracket 1 + \sum_{j < i} \lambda'_i, \sum_{j \leq i} \lambda'_i \right\rrbracket = \left\llbracket 1 + \sum_{j < i} \lambda_i, |s_t| + \sum_{j \leq i} \lambda_i \right\rrbracket = B_t \cup \left\llbracket 1 + \sum_{j \leq t} \lambda_j,  |s_t| + \sum_{j\leq t} \lambda_j \right\rrbracket$$
    as $\lambda'_t = \lambda_t + |s_t|$ and $\lambda'_i = \lambda_i$.
    Finally, if $i > t$, then 
    $$B'_t = \left\llbracket 1 + \sum_{j < i} \lambda'_i, \sum_{j \leq i} \lambda'_i \right\rrbracket = \left\llbracket 1 + |s_t| + \sum_{j < i} \lambda_i, |s_t| + \sum_{j \leq i} \lambda_i \right\rrbracket = \{x + |s_t| \mid x \in B_i\}.$$

    We have that $T_{\lambda'}(x) = x + \sum_{j > i } \lambda'_j - \sum_{j < i } \lambda'_j$. We factorize $\lambda'_t$ to obtain
    $$T_{\lambda'}(x) = \begin{cases}
        x +  \sum_{t > j > i }  \lambda'_j + \sum_{j > t } \lambda'_j - \sum_{j < i }  \lambda'_j + \lambda'_t & \text{ if } i < t \\
        x + \sum_{j > i } \lambda'_j- \sum_{j < i }  \lambda'_j & \text{ if } i = t \\
        x + \sum_{j > i } \lambda'_j - \sum_{j < t }  \lambda'_j- \sum_{t < j < i }  \lambda'_j - \lambda'_t & \text{ if } i > t.
        \end{cases}$$
    Replacing $\lambda'_i$ by $\lambda_i$ if $i \neq t$ and $\lambda'_t$ by $\lambda_t + |s_t|$ give that
    \begin{align*}
        T_{\lambda'}(x) & = \begin{cases}
        x + \sum_{t > j > i }  \lambda_j + \sum_{j > t } \lambda_j - \sum_{j < i }  \lambda_j + \lambda_t + |s_t| & \text{ if } i < t \\
        x + \sum_{j > i } \lambda_j - \sum_{j < i }  \lambda_j & \text{ if } i = t \\
        x + \sum_{j > i } \lambda_j - \sum_{j < t }  \lambda_j- \sum_{t < j < i }  \lambda_j + \lambda_t -|s_t| & \text{ if } i > t,
        \end{cases} \\
        & = \begin{cases}
        x + s_i + |s_t| & \text{ if } i < t \\
        x + s_i         & \text{ if } i = t\\
        x + s_i  -|s_t| & \text{ if } i > t.
        \end{cases}
    \end{align*}
    Therefore, the circular factors of length 2 of $T_{\lambda'}$ are
    $x \cdot (x + s_i + |s_t|)$ if $x \in B_i$ and $i < t$, $x \cdot (x + s_i)$ if $ x \in B_t$ and $x \cdot (x + s_i - |s_t|)$, otherwise. 

    Now, it remains to find the circular factors of length 2 of $\psi(T_{\lambda})$.
    There are two cases to analyze depending on the translation parameter $s_t$.

    Let $x$ be an integer in $\set[n]$. There exist an integer $i$ such that $x \in B_i$. 

    Suppose that  $s_t$ is positive. Thus, we have that $k = \sum_{j \leq t} \lambda_j$.

    If $i < t$, we have that $x \in \set[\sum_{j < t} \lambda_j]$, then $x < k$. Also, we have that $T_{\comp{r}}(x) > k+|s_t|$, since $\sum_{j > t} \lambda_j = s_t + \sum_{j < t} \lambda_j$.
    Therefore, we obtain that 
    $\psi_t(x \cdot T_{\lambda}(x)) = x \cdot (T_{\lambda}(x) + |s_t|) = x \cdot (x + s_i + |s_t|)$.
    Since $i < t$, we know that $x \in B_i$  implies that $x \in B'_i$ and $T_{\lambda'}(x) = x +s_i + |s_t|$ , the element $x$ is sent to the same element in both \DIET s , as desired.
    
    If $i = t$, we have that $x \in \llbracket 1+\sum_{j < t} \lambda_j ,\sum_{j \leq t} \lambda_j\rrbracket$, then $x < k$. Moreover, we have $T_{\comp{r}}(x) \in \llbracket \sum_{j< t} \lambda_j +1, k+|s_t|\rrbracket$, since $\sum_{j > t} \lambda_j = s_t + \sum_{j < t} \lambda_j$.
    Therefore, we obtain that 
    $$\psi_t(x \cdot T_{\lambda}(x)) = \begin{cases}
        x \cdot (x + s_i) & \text{ if } x < k \\
        x \cdot (x + s_i ) \cdot (x + s_i + |s_t|) & \text{otherwise.} 
        \end{cases}$$
        Which means that all elements of $B_t$ are sent to $(x + s_i)$, as desired. 

        If $i > t$, we have that $x \in \llbracket 1+\sum_{j \leq t} \lambda_j , n\rrbracket$, then $x > k$. Moreover, we have that $T_{\lambda}(x) < k + |s_t|$, since $\sum_{j \leq t} \lambda_j > |s_t| + \sum_{j > t} \lambda_j$.
        Therefore, we obtain that
        \begin{align}~\label{eq:cases}
        \psi_t(x\cdot T_{\lambda}(x)) = \begin{cases}
        x \cdot (x + |s_t|) & \text{ if } x \leq k+|s_t| \text{ and }  \sigma(x) \leq k,\\
        x \cdot (x + |s_t|) \cdot T_{\lambda}(x) \cdot (T_{\lambda}(x) + |s_t|)& \text{ if } x \leq k+|s_t| \text{ and }  \sigma(x) > k,\\
        (x + |s_t|) \cdot T_{\lambda}(x)& \text{ if } x > k+|s_t| \text{ and }  T_{\lambda}(x) \leq k,\\
        (x + |s_t|) \cdot T_{\lambda}(x) \cdot (T_{\lambda}(x) + |s_t|) & \text{ if } x > k+|s_t| \text{ and }  T_{\lambda}(x) > k.\\
        \end{cases}
    \end{align}
    Hence, we have that every element between $k+1$ and $k+|s_t|$ is sent to $x + s_i$. Thus, all elements of $B'_t$ are sent to the correct value.
    Moreover, every element of $B'_i$ with $i > t$ are of the form $x + |s_t|$ corresponding to the last two cases in Equation~\eqref{eq:cases}. The circular factors of length 2 are $(x + |s_t|) \cdot T_{\lambda}(x) = (x + |s_t|)\cdot (x + s_i)$, which can be rewritten to $y \cdot (y + s_i - |s_t|)$ if we substitute $y$ to $x + |s_t|$, as desired.

    If $s_t$ is negative, a similar case-by-case analysis shows the equality.
\end{proof}

Now, we are ready to  prove Lemma~\ref{l:diffcycles}.

\begin{proof}[of Lemma~\ref{l:diffcycles}]
    By Lemma~\ref{l:morphism}, we have that
    $$\psi_t(T_{\comp{r}}) = T_{\compplus},$$
    for all $t \in \set$. As a substitution on a set of circular words preserves the number of circular words, we have that
    $$\gamma(T_{\comp{r}}) = \gamma(\psi_t\comp{r}) = \gamma(T_{\compplus})$$ as claimed.
\end{proof}

\begin{theorem}~\label{thm:algo}
    The function $\cycle$ counts the orbits of the \DIET \ $T_{\lambda}$.
\end{theorem}

\begin{proof}
    By Lemma~\ref{l:existence}, we know there exists an integer $t \in \set[r]$ such that $\lambda_t \geq |s_t|$ for any composition $\lambda$. Also, $\compmoins$ and $\compenlever$ are both compositions, since $\lambda_t - |s_t| \geq 0$. Hence, the function $f$ is well-defined.
   
    We show case by case that the recursive step of the function $f$ preserves the number of orbits. 
    First, $\gamma(\lambda_1) = \lambda_1$, since $T_{\lambda_1}(x) = x$ for all $x \in \set[\lambda_1]$.
    Secondly, if $|s_t| = 0$, we know that $\gamma\comp{r} = \lambda_t + \gamma\compenlever$ by Lemma~\ref{l:addcycles}. 
    Thirdly, if $|s_t| = \lambda_t$, we know that 
    $\gamma\comp{r} = \gamma\compmoins$ by Lemma~\ref{l:diffcycles}. However, $\lambda_t = |s_t|$ implies $\compmoins = (\lambda_1,\dots,\lambda_{t-1},0,\lambda_{t+1},\dots,\lambda_r) = \gamma\compenlever$ 
     as desired.
     Finally, if $\lambda_t > |s_t|$, the equality
     $\gamma\comp{r} = \gamma\compmoins$ follows directly from Lemma~\ref{l:diffcycles}. 
     Thus, the function $\cycle$ counts the number of orbits of $T_{\lambda}$.
\end{proof}

\begin{corollary}
    A symmetric discrete interval exchange is minimal if and only if $\cycle\comp{r} = 1$.
\end{corollary}

\section{Tree of circular compositions}~\label{sec:tree}

How can we use the function $f$ to enumerate the compositions defining minimal \DIET? 
For minimal symmetric discrete $2$-interval exchanges, 
solutions are given by the Raney tree~\cite{BdL1997,R1973} (also called Calkin-Wilf tree~\cite{CW2000}) or the Stern-Brocot tree~\cite{B1861,S1858}.
A composition $\lambda$ is called \emph{circular} if the \DIET \ $T_{\lambda}$ is minimal. 
We propose a generalization of the Raney tree to enumerate circular compositions of any lengths.   

The \emph{Raney tree} is a complete infinite binary tree  whose nodes are reduced fractions, described recursively by the following rules :
\begin{enumerate}[$\bullet$]
    \item the root of the tree is the fraction $\frac{1}{1}$;
    \item each vertex $\frac{i}{j}$ has two children: the left child is the fraction $\frac{i}{i+j}$ and the right child is the $\frac{i+j}{j}$.
\end{enumerate}
\begin{figure}
    \begin{center}
    \includegraphics[scale=1]{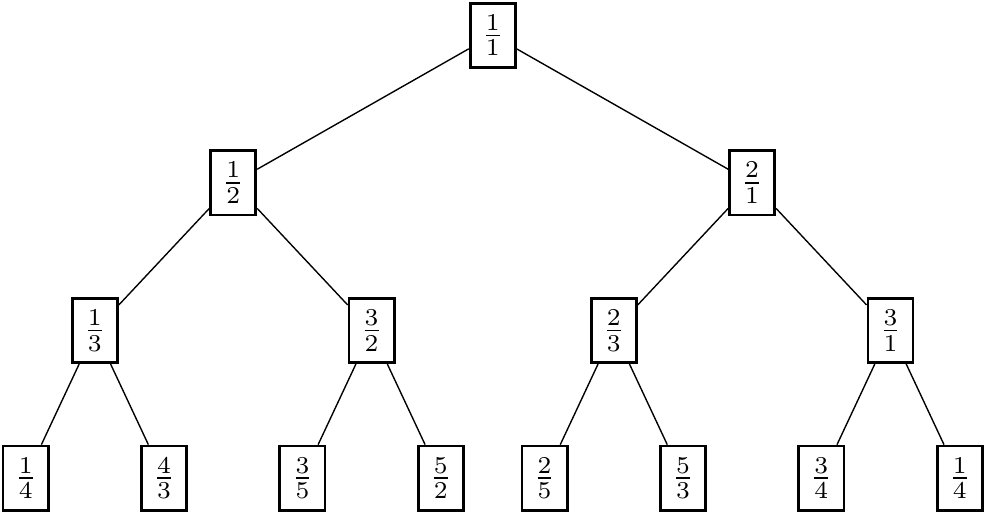}
    \caption{Raney tree}
    \label{fig:arbre_fractions}
    \end{center}
\end{figure}
Figure~\ref{fig:arbre_fractions} shows the first few levels of the Raney tree.
This tree contains each reduced positive rational number exactly once as shown by Berstel and de Luca~\cite{BdL1997}, and by Calkin and Wilf~\cite{CW2000}. Thus it contains all pairs of coprime natural numbers.
Therefore, the Raney tree contains every circular composition of length 2 exactly once.

Observe that applying the function $f$ to any vertex in the Raney tree gives the path from that vertex to the root $1/1$ in the Raney tree.
Therefore, we can use the function $f$ to generalize the Raney tree. The root of this generalized tree is still $(1,1)$, since it is an extension of the Raney tree. A circular composition $\comp{r}$ has children of the form
$\compplus$ for all $t \in \set[r]$ and $(\lambda_1,\dots, \lambda_t, |\delta_t|, \lambda_{t+1}, \dots, \lambda_r )$ for all $t \in \llbracket 0, r \rrbracket$,
where $\delta_t = n - 2\sum_{j \leq t} \lambda_j$.
Figure~\ref{fig:exempletype} shows all the possible children of the composition $(1,1,2,4)$.
Note that $|s_t|$ cannot be zero, otherwise the composition is not circular by Lemma~\ref{l:addcycles}.
Moreover, the children of $\comp{r}$ of the form $\compplus$ are circular by Lemma~\ref{l:diffcycles}. 
There are two problems with the children of $\comp{r}$ of the form $(\lambda_1,\dots,\lambda_{t},|\delta_t|,\lambda_{t+1},\dots,\lambda_r)$.
First, if $\delta_t = 0$ (e.g. $\delta_3$ in Figure~\ref{fig:exempletype}), the result is the same \DIET. We want to avoid this case as it does not  add any new information.
Secondly, the resulting composition has two parents.  
For example, the parents of the circular composition $(1,1,2,4)$ are the compositions $(1,1,4)$ ($\delta = (6,4,\underline{2},-6)$) and $(1,1,2)$ ($\delta = (4,2,0,\underline{-4})$) (as shown in Figure~\ref{fig:coincide}).

\begin{figure}
    \begin{center}
    \begin{minipage}{0.49\textwidth}
    \centering
    \includegraphics[scale=0.75]{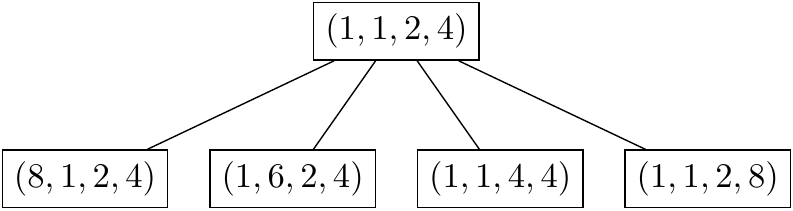}
    \subcaption{Children of type 1}
    \end{minipage}
    \begin{minipage}{0.49\textwidth}
    \centering
    \includegraphics[scale=0.75]{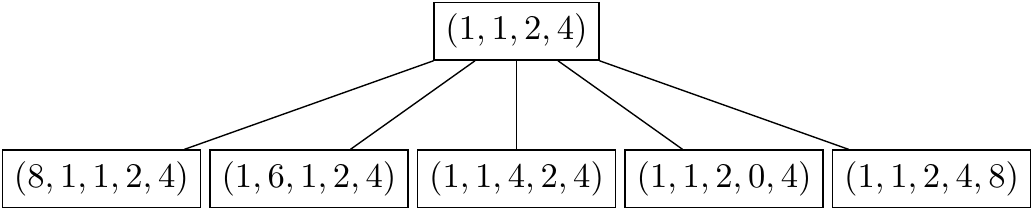}
    \subcaption{Children of Type 2 (Note that $(1,1,2,0,4)$ is not a composition)}
    \end{minipage}
    \caption{Children of Type 1 and 2 of the composition $(1,1,2,4)$ whose translation vector is $(7,5,2,-4)$ and $\delta = (8,6,4,0,-8)$}
    \label{fig:exempletype}
    \end{center}
\end{figure}

\begin{figure}
    \begin{center}
    \includegraphics[scale=1]{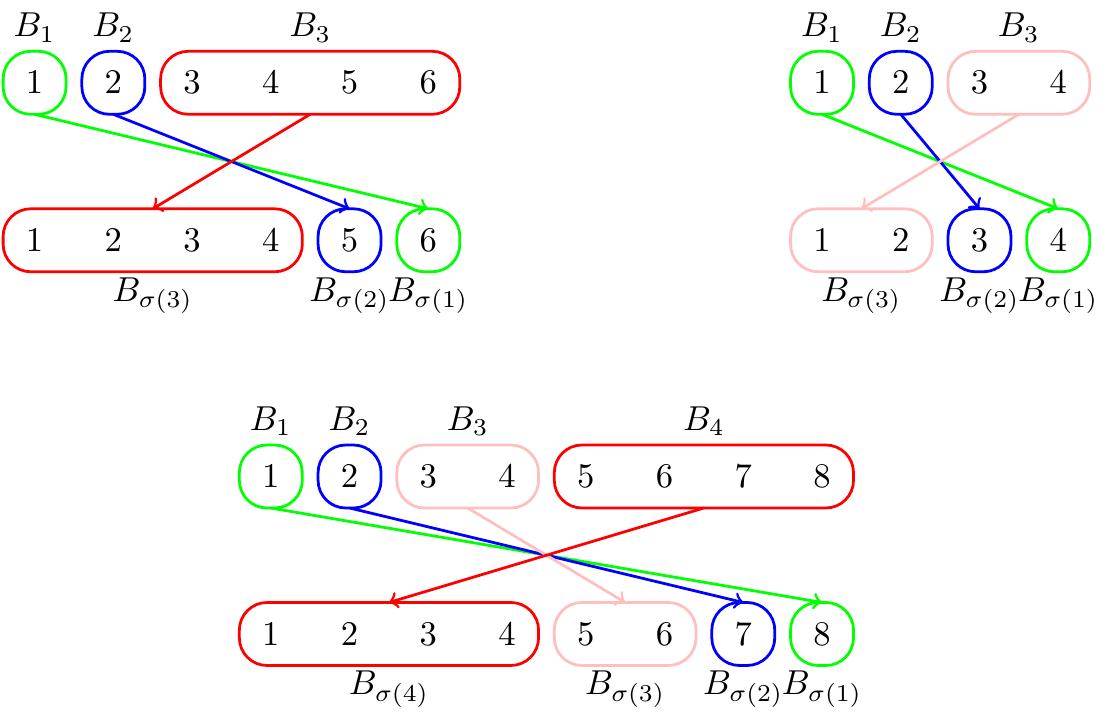}
    \caption{The top right discrete interval is $(1,1,4)$ and the top left one is $(1,1,2)$ both can be the parent of $(1,1,2,4)$. The $\delta$ of $(1,1,4)$ is $(6,4,2,-6)$ and the $\delta$ of $(1,1,2)$ is $(4,2,0,-4)$. Hence, $(1,1,|\delta_2|,2) = (1,1,2,4) = (1,1,2,|\delta_4|)$}
    \label{fig:coincide}
    \end{center}
\end{figure}

Adding $\delta_t$ preserves the number of orbits of the symmetric discrete interval exchange, since the $(t+1)$-th translation parameter of the composition $(\lambda_1,\dots, \lambda_t, |\delta_t|, \lambda_{t+1}, \dots, \lambda_r )$ is by definition $ s_{t+1} = \sum_{j > t} \lambda_t - \sum_{j \leq t} \lambda_t = n - 2 \sum_{j \leq t} = \delta_t$. 
By Lemma~\ref{l:diffcycles}, these children of $\comp{r}$ are also circular. 
Compositions constructed by the two rules have two parents, since adding $\delta_t$ or $\lambda_t$ if $\delta_t > 0$ (resp. $\lambda_{t+1}$ if $\delta_t < 0$) produces the same composition (see Figure~\ref{fig:coincide}). 
To ensure that every child is a composition and that each composition has a single parent, we add a condition on $\delta$.

The \emph{tree of circular compositions} is an infinite tree described by the following recursive rules:
\begin{enumerate}[$\bullet$]
    \item the root of the tree is the composition $(1,1)$;
    \item each vertex $\comp{r}$ has children of the form:
        \begin{enumerate}[{Type }1:]
            \item $\compplus$ for all $t \in \set$;
            \item if $\delta_t > \lambda_{t+1}$ or $-\delta_t > \lambda_t$, $(\lambda_1,\dots, \lambda_t, |\delta_t|, \lambda_{t+1}, \dots, \lambda_r )$ for all $t \in \llbracket 0, r \rrbracket$;
        \end{enumerate}
\end{enumerate}
where $\delta_t = n - \sum_{j \leq t} \lambda_j$.

\begin{figure}
    \begin{center}
    \includegraphics[scale=1]{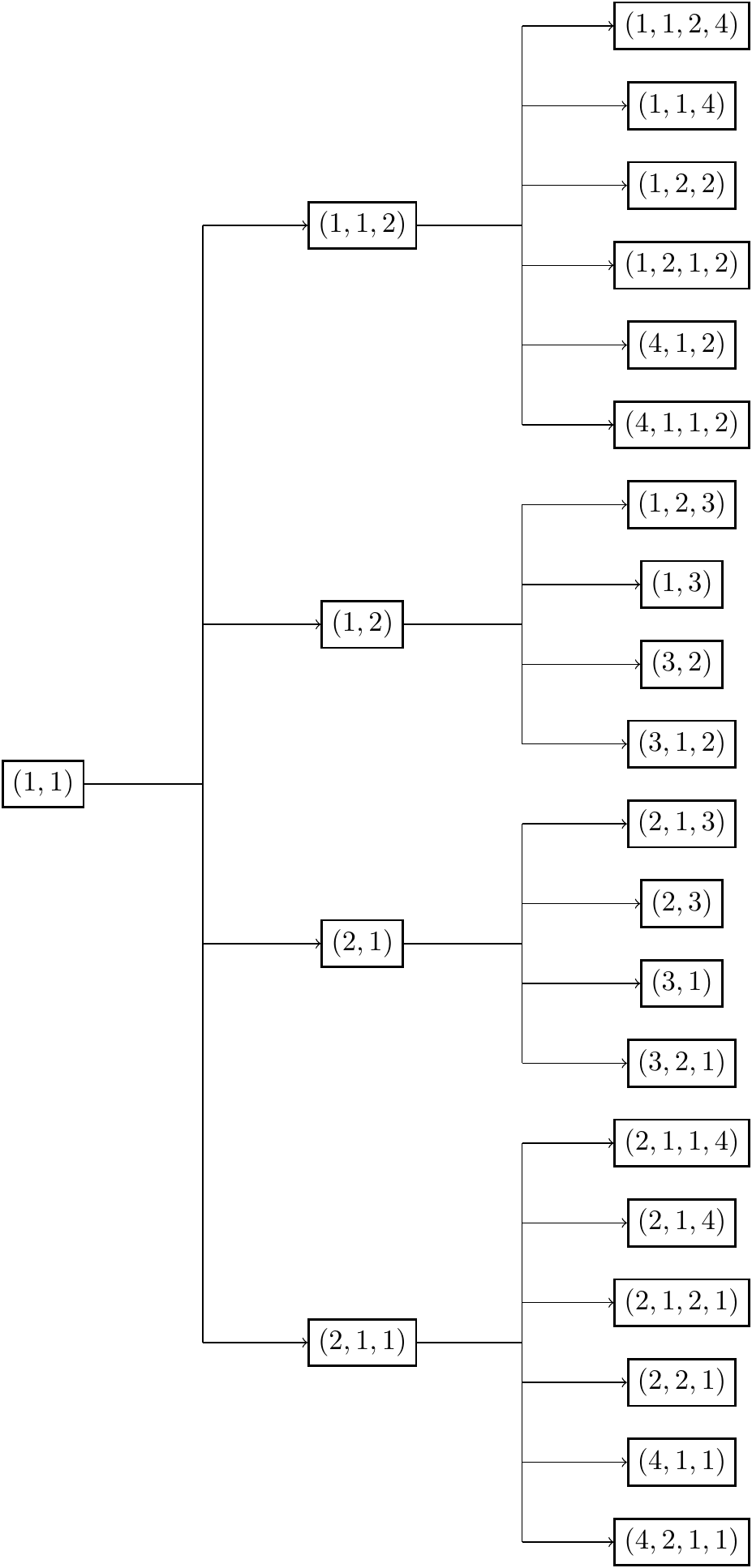}
    \caption{The tree of circular compositions}
    \label{fig:gen_tree}
    \end{center}
\end{figure}

Figure~\ref{fig:gen_tree} shows the first few levels of the tree of circular compositions.
Note that $\delta_0 > \lambda_1$ and $-\delta_r > \lambda_r$, hence the condition is always satisfied in those cases.
Thus, the parent of $(1,1,2,4)$ is $(1,1,2)$, since $-\delta_3 = 4 > \lambda_3$ and $(1,1,4)$ is not its parent since $\delta_3 = 2 < \lambda_4 = 4$ and $-\delta_3 = -2 < \lambda_3 = 2$.
Now we prove that each circular composition is a child of type $1$ or $2$ but not both and has a unique parent.

\begin{lemma}~\label{l:unicite1}
    Let $\comp{r}$ be a circular composition and $t$ an integer in $\set[r]$. 
    If $\lambda_t > |s_t|$, then $\lambda_i < |s_i|$ for all $i \in \set$, $i \neq t$.
\end{lemma}

\begin{proof}
    Let $t$ be an integer in $\set$ such that $\lambda_t > |s_t|$.
    First, we show that $s_i > 0$ if $i < t$ (resp. $s_i < 0$ if $i > t$). 
    Let $i$ be an integer in the set $\set[t-1]$.

    If $s_t > 0$, by the definition of $s_t$ the following inequality is satisfied: 
    $$\sum_{k > t} \lambda_k > \sum_{k<t} \lambda_k.$$
    Since $i < t$ and all $\lambda_k$'s are positive, we have that 
    $$\sum_{k > i} \lambda_k > \sum_{k > t} \lambda_k > \sum_{k<t} \lambda_k > \sum_{k<i} \lambda_k.$$
    Hence $s_i  = \sum_{k > i} \lambda_k -\sum_{k<i} \lambda_k > 0$.

    If $s_t < 0$, by the definition of $s_t$, we know that 
        $$\lambda_t > |s_t| = -s_t = \sum_{k < t} \lambda_k - \sum_{k>t} \lambda_k.$$
        Moreover, adding $\sum_{k > t} \lambda_k$ on both sides, we have that 
        $$ \lambda_t + \sum_{k > t} \lambda_k > \sum_{k<t} \lambda_k.$$
    Since $i < t$ and all $\lambda_k$'s are positive, the previous inequality implies that 
    $$\sum_{k > i} \lambda_k \geq \lambda_t + \sum_{k > t} \lambda_k > \sum_{k<t} \lambda_k > \sum_{k<i} \lambda_k.$$
    Hence, $s_i = \sum_{k > i} \lambda_k -\sum_{k<i} \lambda_k > 0$ as claimed.

    Similar arguments show that, if  $i > t$, then $s_i < 0$.

    Now we can show that $\lambda_i < |s_i|$ if $i \neq t$. There are two cases to consider.
    
    Suppose that $s_t > 0$. 
    If $i < t$, the translation parameter $s_i$ can be expressed as a sum of $s_t$ and some $\lambda_k's$ given below:
    \begin{align*}
        s_t + \lambda_t + 2\sum_{i < j < t} \lambda_j + \lambda_i =&  \sum_{j > t} \lambda_j - \sum_{j<t} \lambda_j + \lambda_t + 2\sum_{i < j < t} \lambda_j + \lambda_i   \\
        =& \sum_{j > i} \lambda_j - \sum_{j < i } \lambda_j \\
        =& s_i.
    \end{align*}
    Moreover, we have that $ s_i = s_t + \lambda_t + 2\sum_{i < j < t} \lambda_j + \lambda_i > \lambda_i$ 
    since $s_t$ is positive. Thus, $|s_i| = s_i > \lambda_i$ since $i < t$.

    If $i > t$, the translation parameter $s_i$ can be expressed by $s_t$ and some $\lambda_k$ as below:
    \begin{align*}
        s_t - \lambda_t - 2\sum_{t < j < i} \lambda_j - \lambda_i =& \sum_{j > t} \lambda_j - \sum_{j < t} \lambda_j - \lambda_t - 2\sum_{t < j < i} \lambda_j - \lambda_i \\ 
        =& \sum_{j > i} \lambda_j - \sum_{j < i } \lambda_j \\
        =& s_i.
    \end{align*}
    Since $i > t$, we have that $s_i <0$, hence $|s_i| = -s_t + \lambda_t + 2\sum_{t < j < i} \lambda_j + \lambda_i$ which is a sum of positive integers since $\lambda_t - s_t > 0$. Hence, $|s_i| > \lambda_i$ as claimed.
    
    Similar arguments show that the lemma also holds in the case where $s_t$ is negative.
\end{proof}

\begin{lemma}~\label{l:unicite2}
    Let $\comp{r}$ be a circular composition and $t$ be an integer in $\set[r]$.
    If $\lambda_t = |s_t|$, then there exists an integer $d \in \{t-1,t+1\}$ such that $\lambda_d = |s_d|$ and $\lambda_i < |s_i|$ for all $i \in \set - \{t,d\}$. Moreover $\lambda_d \neq \lambda_t$.
\end{lemma}

\begin{proof}
    Let $\comp{r}$ be a circular composition such that $\lambda_t = |s_t|$.
    First, we check that $\lambda_d$ exists.

    If $s_t > 0$, then 
    $$\lambda_t = \sum_{k>t} \lambda_k - \sum_{k < t} \lambda_k \quad  \Rightarrow \quad 
    \lambda_{t+1} = \sum_{k<t+1} \lambda_k - \sum_{k > t +1} \lambda_k = |s_{t+1}|.$$

    Similarly, if $s_t < 0$, we have that $\lambda_{t-1} = |s_{t-1}|$.

    Secondly, we have to show that $\lambda_t \neq \lambda_d$.
    Suppose that we are in the case $d = t+1$.
    Suppose further that $\lambda_t = \lambda_d$.
    Since $\lambda_t = |s_t|$,  we have that 
    $\gamma\comp{r} = \gamma\compenlever$ by Theorem~\ref{thm:algo}.
    Moreover, $\lambda_t = \sum_{k > t} \lambda_k - \sum_{k< t} \lambda_k$ and $\lambda_t = \lambda_{t+1}$ imply that $0 = \sum_{k > t+1} \lambda_k - \sum_{k < t} \lambda_k$.
    Hence, $ \gamma\compenlever = \lambda_{t+1} + \gamma(\lambda_1,\dots,\lambda_{t-1},\lambda_{t+2},\dots,\lambda_r)$ by Lemma~\ref{l:addcycles}.
    Thus $\gamma\comp{r}$ is greater than 1 since $\lambda_{t+1} > 0$, a contradiction. 
    The case where $d = t-1$ is similar.

    Finally, we have to check that $\lambda_i < |s_i|$ if $i \neq d$ and $i \neq t$.
    Suppose that $s_t > 0$, then $d = t+1$.
    If $i < t$, then we have that 
        $\lambda_t  = s_t 
                   = \sum_{k > t} \lambda_k - \sum_{k < t} \lambda_k$.
    We isolate $\lambda_i$ in this equation to obtain
    $\lambda_i = -\sum_{j < i} \lambda_j - \sum_{i+1 < j \leq t} \lambda_j + \sum_{j > t} \lambda_j$. 
    Moreover,  we have that
    \begin{align*}
        \lambda_i & = -\sum_{k<i} \lambda_k - \sum_{i+1 < k \leq t} \lambda_k + \sum_{k > t} \lambda_k + s_i - s_i\\
        & = s_i - 2 \sum_{i+1 < k \leq t} \lambda_k \\
        & < s_i
    \end{align*}
    since all $\lambda_k$'s are positive.

    Similar arguments show that if $i > t$, then $\lambda_i < |s_i|$.
    To complete the proof, we use a similar argument to show that $\lambda_i < |s_i|$ if $s_t < 0$.
\end{proof}

\begin{corollary}~\label{cor:uni}
    Let $\comp{r}$ be a circular composition. If the composition is not $(1,1)$, then it has a unique parent in the tree of circular composition. 
\end{corollary}

\begin{proof}
    There exists a $t \in \set[r]$ such that $\lambda_t \geq |s_t|$ by Lemma~\ref{l:existence} and $s_t$ is positive, since the composition is circular by Lemma~\ref{l:addcycles}.
    Therefore, the composition $\comp{r}$ is a child of type 1 or type 2 of the composition $\lambda'$.

    Suppose that $\comp{r}$ is a child of type 1.
    We know that $\lambda_i < |s_i|$ for all $i \in \set[r] - \{t\}$ by Lemma~\ref{l:unicite1}.
    Thus $\comp{r}$ is not a child of type 2 and the only valid position to reduce the composition is $t$. 
    So the parent of $\comp{r}$ is $\lambda' = \compmoins$.

    Suppose that $\comp{r}$ is a child of type 2.
    We know that $\lambda_i < |s_i|$ for all $i \in \set[r] - \{t,d\}$.
    Thus $\comp{r}$ is not a child of type 1, since any $\lambda_i \leq |s_i|$.
    If $\delta_t > \lambda_{t+1}$ or $-\delta_t > \lambda_{t}$, then the parent of $\comp{r}$ is $\lambda' = \compenlever$, otherwise it is $\lambda' = (\lambda_1,\dots,\lambda_t,\lambda_{t+2},\dots, \lambda_r)$ if $\delta_{t} > 0$, or 
    $ \lambda' = (\lambda_1,\dots,\lambda_{t-2},\lambda_{t},\dots, \lambda_r)$ if $\delta_t < 0$.

\end{proof}

\begin{theorem}
	Every circular composition appears exactly once on the tree of circular composition.
\end{theorem}

\begin{proof}
    First, we show that every composition in the tree is circular.
    The \DIET \ $T_{(1,1)}$ is the permutation $(1, 2)$ which is obviously minimal.
    Let $\comp{r}$ be a composition that appears in the tree which is not associated with a circular permutation and such that all compositions appearing in previous levels of the tree are circular.
    There exists an integer $t \in \set$ such that $\lambda_t \geq |s_t|$ and 
    $$   \gamma\compmoins = \gamma\comp{r} > 1,$$
    by Theorem~\ref{thm:algo}.
    The composition $\compmoins$ appears in a previous level of the tree, which is a contradiction.
    Thus, all the compositions in the tree are circular.

    Secondly, we prove that all circular compositions appear in the tree.
    The composition $(1,1)$ is the root of the tree. 
    Let $S$ be the set of all circular compositions which are not in the tree and $\comp{r}$  the smallest composition in $S$ ordered them as in Section~\ref{sec:counting}.
    There exists a natural number $t \in \set$ such that $\lambda_t \geq |s_t|$ and 
    $$\gamma\compmoins = \gamma\comp{r} $$
    by Lemma~\ref{l:existence}. Therefore the composition $\compmoins$ cannot be an element of the tree, otherwise $\comp{r}$ is also an element of the tree contradicting the  minimality of $\comp{r}$.

    Thirdly, we show that no circular composition occurs more than once in the tree. 
    Let $\comp{r}$ be the smallest composition which appears more than once in the tree. 
    By Corollary~\ref{cor:uni}, we know that $\comp{r}$ has a unique parent $\lambda'$. Then $\comp{r}$ is the child of two distinct vertices both labelled by $\lambda'$, contradicting the minimality of $\comp{r}$. 
    Therefore, each circular composition appears only once in the tree.
\end{proof}

\section{Cyclic type}~\label{sec:type}

The \emph{cyclic type} of a \DIET \ is a partition $\ell_1^{\alpha_1}\dots\ell_k^{\alpha_k}$ where $\ell_1 > \ell_2 > \dots > \ell_k$, $\ell_i$ is the length of an orbit of $T_{\lambda}$ and $\alpha_i$ is the number of orbits of length $\ell_i$.
For example, the cyclic type of $T_{(3,5,4,2)}$ is the partition $8^13^2$. The cyclic type of symmetric discrete 2-interval exchanges is already known.

\begin{lemma}(folklore)~\label{l:folk}
    Let $(\lambda_1,\lambda_2)$ be a composition. The cyclic type of the \DIET \ $T_{\comp{r}}$ is the partition 
    $$\left(\frac{\lambda_1+\lambda_2}{\gcd(\lambda_1,\lambda_2)}\right)^{\gcd(\lambda_1,\lambda_2)}.$$
\end{lemma}

The cyclic type of \DIETtrois \ can also be described by the $\gcd$ function.

\begin{lemma}~\label{l:s_gcd}
    Let $\comptrois$ be a composition. The integers $s_1,\ s_2$ and $s_3$ are multiples of $\gcd(\lambda_1+\lambda_2,\lambda_2+\lambda_3)$.
\end{lemma}

\begin{proof}
    The result is straightforward for $s_1 = \lambda_2+\lambda_3$ and $s_3 = -\lambda_1-\lambda_2$.
    Moreover, subtracting $\lambda_2+\lambda_3$ from $\lambda_1+\lambda_2$, we obtain that $\gcd(\lambda_1+\lambda_2, \lambda_2+\lambda_3) = \gcd(\lambda_1-\lambda_3,\lambda_2+\lambda_3)$.
    Thus $s_2 = |\lambda_3-\lambda_1|$ is a multiple of $\gcd(\lambda_1+\lambda_2,\lambda_2+\lambda_3)$.
\end{proof}

Now, we can describe the number of orbits of a \DIETtrois \ using the $\gcd(\lambda_1+\lambda_2,\lambda_2+\lambda_3)$.

\begin{theorem}~\label{thm:orbit}
    The orbit of $x \in \set[n]$ in the \DIET \ $T_{\comptrois}$ is the set
    $$\left\{x + kd \, \middle\vert \, k \in \N \, \mathrm{ and }  \, \frac{1-x}{d} \leq k \leq \frac{n-x}{d}\right\},$$
    where $d = gcd(\lambda_1+\lambda_2,\lambda_2+\lambda_3)$.
\end{theorem}

For example, the cyclic decomposition of the \DIET of $T_{(9,1,4)}$ is $(1,6,11)(2,7,12)(3,8,14)(4,9,14)(5,10)$. Therefore, the orbit of $9$ is $\{9 + 5k | -8/5 \leq k \leq 5/5\} = \{4,9,14\}$ .

\begin{proof}
    Let $d$ be the $\gcd(\lambda_1+\lambda_2, \lambda_2+\lambda_3)$ and $x$ be an integer in $\set[n]$.
    First, we show by induction that $T^{(k)}_{\comptrois}(x) = x + id$. We know that $T^0_{\comptrois}(x) = x + 0\cdot d$ and $T^1_{\comptrois}(x) = x + s_i = x + id$ since $s_i$ is a multiple of $d$ as shown in Lemma~\ref{l:s_gcd}.
    Let computes $$T^{(k+1)}_{\comptrois}(x) = T(T^{k}_{\comptrois}(x)) = T(x+id) = x + id + s_i$$ by induction. Moreover, $s_i$ is a multiple of $d$ by Lemma~\ref{l:s_gcd}. Hence, we have that $T^{(k+1)}_{\comptrois}(x) = x+ jd$.   

    By the definition of $T_{\comptrois}$, we have that $1 \leq x +kd \leq n$ which gives that $(1-x)/d \leq k \leq (n-x)/d$.
\end{proof}

\begin{corollary}~\label{cor:cycle_type_3}
    Let $\comptrois$ be a composition of $n = \lambda_1 + \lambda_2 + \lambda_3$. Take $d = \gcd(\lambda_1 + \lambda_2, \lambda_2 +\lambda_3)$. 
    Let $q$ and $r$ be integers with $0 \leq r < d $ such that $n = qd + r$.  
    The cyclic type of  $T_{\comptrois}$ is $q^{d}$ if $r = 0$ and $(q+1)^r q^{d-r}$ otherwise. Moreover $d = \gamma(T_{\comptrois})$.
\end{corollary}

\begin{proof}
    By Theorem~\ref{thm:orbit}, we have that the orbit of $x$ is the set
    $$\left\{x + kd \, \middle\vert \, k \in \N \, \mathrm{ and }  \, \frac{1-x}{d} \leq k \leq \frac{n-x}{d}\right\}.$$ Therefore, every element in the orbit of $x$ is congruent to $x$ modulo $d$. Hence, the number of orbit is the number of classes modulo $d$, which is exactly $d$. The cardinality of $\gamma(x)$ is $[n/d]$ if $x \bmod d < r$ or $[n/d] +1$ if $x \bmod d \geq r$.
\end{proof}

The \DIET s $T_\lambda$ where $\lambda$ is a composition of length 2 or 3 has simple cyclic type.
For symmetric discrete interval exchanges of length 4 or more, no such description is known. As a matter of fact, it does not seem to have any relation between the length of the orbits in the cyclic type. However computer testing suggests that if $\ell_1^{\alpha_1}\dots\ell_k^{\alpha^k}$ is the cyclic type of $T_{\comp{r}}$, then $k \leq \lceil \frac{r-1}{2}\rceil$. 

\textbf{Acknowledgement.} I would like to thank Christophe Reutenauer, Alexandre Blondin Mass\'e, Vincent Delecroix and Michael Rao for their suggestions and helpful comments. This work was supported by NSERC graduate scholarship.

\nocite{*}
\bibliographystyle{plain}
\bibliography{bibliographie}

\begin{thebibliography}{10}

\bibitem{A1963}
V.~I. Arnold.
\newblock Small denominators and problems of stability of motion in classical
  and celestial mechanics.
\newblock {\em Uspehi Mat. Nauk}, 18(6 (114)):91--192, 1963.
\newblock translated in Russian Math. Surveys {\bf 18}:86--194, 1963.

\bibitem{A2004}
V.~I. Arnold.
\newblock {\em Arnold's problems}.
\newblock Springer-Verlag, Berlin; PHASIS, Moscow, 2004.
\newblock Translated and revised edition of the 2000 Russian original, With a
  preface by V. Philippov, A. Yakivchik and M. Peters.

\bibitem{BdL1997}
Jean Berstel and Aldo de~Luca.
\newblock Sturmian words, {L}yndon words and trees.
\newblock {\em Theoret. Comput. Sci.}, 178(1-2):171--203, 1997.

\bibitem{B1861}
Achille Brocot.
\newblock Calcul des rouages par approximation, nouvelle m\'ethode.
\newblock {\em Revue Chonométrique}, 3:186--194, 1861.

\bibitem{CW2000}
Neil Calkin and Herbert~S. Wilf.
\newblock Recounting the rationals.
\newblock {\em Amer. Math. Monthly}, 107(4):360--363, 2000.

\bibitem{FZ2010}
S\'{e}bastien Ferenczi and Luca~Q. Zamboni.
\newblock Structure of {$K$}-interval exchange transformations: induction,
  trajectories, and distance theorems.
\newblock {\em J. Anal. Math.}, 112:289--328, 2010.

\bibitem{FZ2013}
S\'ebastien Ferenczi and Luca~Q. Zamboni.
\newblock Clustering words and interval exchanges.
\newblock {\em J. Integer Seq.}, 16(2):Article 13.2.1, 9, 2013.

\bibitem{KL2017}
Anna Karnauhova and Stefan Liebscher.
\newblock Connected components of meanders: {I}. {B}i-rainbow meanders.
\newblock {\em Discrete Contin. Dyn. Syst.}, 37(9):4835--4856, 2017.

\bibitem{L2015}
S{\'e}bastien {Labb{\'e}}.
\newblock {$3$-dimensional Continued Fraction Algorithms Cheat Sheets}.
\newblock {\em ArXiv e-prints}, November 2015.

\bibitem{LdR1994}
Artur~O. Lopes and Luiz Fernando~C. da~Rocha.
\newblock Invariant measures for {G}auss maps associated with interval exchange
  maps.
\newblock {\em Indiana Univ. Math. J.}, 43(4):1399--1438, 1994.

\bibitem{O1966}
V.~I. Oseledec.
\newblock The spectrum of ergodic automorphisms.
\newblock {\em Dokl. Akad. Nauk SSSR}, 168:1009--1011, 1966.

\bibitem{PR2008}
Igor Pak and Amanda Redlich.
\newblock Long cycles in {$abc$}-permutations.
\newblock {\em Funct. Anal. Other Math.}, 2(1):87--92, 2008.

\bibitem{R1973}
George~N. Raney.
\newblock On continued fractions and finite automata.
\newblock {\em Math. Ann.}, 206:265--283, 1973.

\bibitem{R1979}
G\'erard Rauzy.
\newblock {\'E}changes d'intervalles et transformations induites.
\newblock {\em Acta Arith.}, 34(4):315--328, 1979.

\bibitem{R2016}
Christophe Reutenauer.
\newblock Personnal communication, November 2016.

\bibitem{S2000}
Fritz Schweiger.
\newblock {\em Multidimensional {C}ontinued {F}raction}.
\newblock Oxford Univ. Press, New York, 2000.

\bibitem{S1858}
M.~A. Stern.
\newblock Über eine zahlentheoretische funktion.
\newblock {\em j. reine angew. Math.}, 55:193--220, 1858.

\end{thebibliography}
\label{sec:biblio}

\end{document}